\documentclass[12pt]{amsart}
\usepackage[left=25mm,right=25mm,top=25mm,bottom=25mm]{geometry}
\usepackage{palatino}
\usepackage{amsfonts}
\usepackage{amssymb}
\usepackage{amscd}
\usepackage{hyperref}
\usepackage{graphicx}
\usepackage{tikz}
\usepackage{amsmath,amsfonts,amssymb}
\usetikzlibrary{arrows}
\usepackage{color}
\definecolor{ashgrey}{rgb}{0.7, 0.75, 0.71}
\definecolor{oxfordblue}{rgb}{0.0, 0.13, 0.28}
\definecolor{armygreen}{rgb}{0.29, 0.33, 0.13}
\definecolor{bulgarianrose}{rgb}{0.28, 0.02, 0.03}
\definecolor{carnelian}{rgb}{0.7, 0.11, 0.11}
\definecolor{lapislazuli}{rgb}{0.15, 0.38, 0.61}
\definecolor{mediumelectricblue}{rgb}{0.01, 0.31, 0.59}

\newtheorem{thm}{Theorem}[section]
\newtheorem{cor}[thm]{Corollary}
\newtheorem{lem}[thm]{Lemma}
\newtheorem{prop}[thm]{Proposition}
\newtheorem{conj}[thm]{Conjecture}

\theoremstyle{definition}
\newtheorem{dfn}[thm]{Definition}
\theoremstyle{remark}

\numberwithin{equation}{subsection}

\newtheorem{Q}[thm]{Question}

\newtheorem{con}[thm]{Conventions}

\newcommand{\al}{\alpha}

\newcommand{\ka}{\kappa}

\newcommand{\w}{\omega}
\newcommand{\CC}{\mathcal{C}}

\newcommand{\MM}{\overline{\mathcal{M}}}

\newcommand{\CG}{\mathcal{G}}

\newcommand{\CF}{\mathcal{F}}

\newcommand{\CM}{\mathcal{M}}

\newcommand{\E}{\mathbb{E}}

\newcommand{\G}{\mathcal{G}}

\newcommand{\QQ}{\mathbb{Q}}

\newcommand{\D}{\qquad}

\newcommand{\too}{\rightarrow}

\begin{document}

\title[The connection between $R^*(\CC_g^n)$ and $R^*(\CM_{g,n}^{rt})$]
{The connection between $R^*(\CC_g^n)$ and $R^*(\CM_{g,n}^{rt})$}
\author[M. Tavakol]{Mehdi Tavakol}
   \address{IBS Center for Geometry and Physics}
   \email{mehdi@ibs.re.kr}

\begin{abstract}
We study the connection between tautological classes on the moduli spaces $\CC_g^n$ and $\CM_{g,n}^{rt}$.
A conjectural connection between tautological relations on $\CC_g^n$ and $\CM_{g,n}^{rt}$ is described.
The analogue of this conjecture for the Gorenstein quotients of tautological rings is proved.
\end{abstract}   

\maketitle

\section*{Introduction}
Let $\CM_{g,n}^{rt}$ be the moduli space of stable $n$-pointed curves of genus $g>1$ with rational tails.
This space is a partial compactification of the space $\CM_{g,n}$, 
classifying smooth curves of genus $g$ together with $n$ ordered distinct points. 
It sits inside the Deligne-Mumford compactification $\MM_{g,n}$ of $\CM_{g,n}$, which classifies all stable $n$-pointed curves of genus $g$.
We also consider the space $\CC_g^n$ classifying smooth curves of genus $g$ with not necessarily distinct $n$ ordered points.
There is a natural proper map from $\CM_{g,n}^{rt}$ to $\CC_g^n$ which contracts all rational components.
Tautological classes on these spaces are natural algebraic cycles reflecting the geometry of curves.
Their definition and analysis on $\CM_g$ and $\MM_g$ was started by Mumford in the seminal paper \cite{M}.
These were later generalized to pointed spaces by Faber and Pandharipande \cite{FP2}. 
In this short note we study the connection between tautological classes on $\CC_g^n$ and $\CM_{g,n}^{rt}$.
We show that there is a natural filtration on the tautological ring of $\CM_{g,n}^{rt}$ consisting of $g-2+n$ steps. 
A conjectural dictionary between tautological relations on $\CC_g^n$ and $\CM_{g,n}^{rt}$ is presented.
We prove the analogue version of our conjecture for the Gorenstein quotients of tautological rings.
It is well-known that the tautological group of $\CM_{g,n}^{rt}$ is one dimensional in degree $g-2+n$ and it vanishes in higher degrees.
In \cite{GV2} the authors predicted that these should be deductible from the main theorem of Looijenga \cite{L}.
We show that both statements follow from the result of Looijenga on $\CC_g^n$ and a result of Faber. 

\begin{con}
We consider algebraic cycles modulo rational equivalence.
All Chow groups are taken with $\QQ$-coefficients. 
\end{con}

\noindent{\bf Acknowledgments.}
This note was prepared during my stay at 
KdV Instituut voor Wiskunde in the
research group of Sergey Shadrin.
I was supported by the research grant IBS-R003-S1.
Discussions with Carel Faber and Qizheng Yin inspired this study.  
I am grateful to Pierre Deligne for his useful remarks on the preliminary version of this note. 

\section{Tautological algebras}

\subsection{The moduli spaces}

Let $\pi: \CC_g \too \CM_g$ be the universal smooth curve of genus $g>1$.
For an integer $n>0$ the $n$-fold fiber product of $\CC_g$ over $\CM_g$ is denoted by $\CC_g^n$.
It parameterizes smooth curves of genus $g$ together with $n$ ordered points.
We also consider the space $\CM_{g,n}^{rt}$ which classifies stable $n$-pointed curves of genus $g$ with rational tails.
This moduli space classifies stable curves of arithmetic genus $g$ consisting of a unique component of genus $g$.
These conditions imply that all other components are rational.
The marked points belong to the smooth locus of the curve.
Marked points and nodal points are called special.
By the stability condition we require that every rational component has at least 3 special points.
As a result the corresponding moduli point has finitely many 
automorphisms and the associated stack is of Deligne-Mumford type.
Recall that the boundary $\CM_{g,n}^{rt} \setminus \CM_{g,n}$ is a divisor with normal crossings.
It is a union of irreducible divisor classes $D_I$ parameterized by subsets $I$ of the set $\{1, \dots , n\}$
having at least 2 elements. 
The generic point of the divisor $D_I$ corresponds to a nodal curve with 2 components.
One of its components is of genus $g$ and the other component is rational.
The set $I$ refers to the markings on the rational component.

\subsection{Tautological classes}

Tautological classes are natural algebraic cycles on the moduli space.
Their original definition on the space $\CM_g$ and its compactification 
$\MM_g$ is given by Mumford in \cite{M}.
The definition of tautological algebras of moduli spaces of 
pointed curves is due to Faber and Pandharipande \cite{FP2}. 
We will recall the definition below.

\begin{dfn}
Let $g,n$ be integers so that $2g-2+n>0$.
The system of tautological rings is defined to be the set of smallest 
$\QQ$-subalgebras $R^*(\MM_{g,n})$ of the Chow rings $A^*(\MM_{g,n})$ 
which is closed under push-forward via all maps forgetting points and all gluing maps. 
\end{dfn}

\begin{dfn}
The tautological ring $R^*(\CM_{g,n}^{rt})$ of $\CM_{g,n}^{rt}$ is the image of $R^*(\MM_{g,n})$ 
via the restriction map.
\end{dfn}

Recall that there is a proper map $p: \CM_{g,n}^{rt} \too \CC_g^n$ 
which contracts all rational components of the curve.

\begin{dfn}
The tautological ring $R^*(\CC_g^n)$ of $\CC_g^n$ is defined as the image of $R^*(\CM_{g,n}^{rt})$ 
via the push-forward map $p_*:A^*(\CM_{g,n}^{rt}) \too A^*(\CC_g^n)$.
\end{dfn}

There is a more explicit presentation of the tautological ring of $\CC_g^n$ and $\CM_{g,n}^{rt}$.
Let $\pi: \CC_g \too \CM_g$ be the universal curve of genus $g$ as before.
Denote by $\w$ its relative dualizing sheaf.
Its class in the Chow group $A^1(\CC_g)$ of $\CC_g$ is denoted by $K$.
This defines the class $K_i$ for every $1 \leq i \leq n$ on $\CC_g^n$.
We also denote by $d_{i,j}$ the class of the diagonal for $1 \leq i<j \leq n$.
The kappa class $\ka_i$ is the push-forward class $\pi_*(K^{i+1})$.
According to the vanishing results proven by Looijenga \cite{L} the class $\ka_i$ is zero when $i>g-2$.
We denote the pull-back of these classes to $\CM_{g,n}^{rt}$ along the contraction map $p$ by the same letters. 
The proof of the following fact is left to the reader:

\begin{lem}
(1) The tautological ring of $\CC_g^n$ is the 
$\QQ$-subalgebra of its Chow ring $A^*(\CC_g^n)$ 
generated by kappa classes $\ka_i$ for $1 \leq i \leq g-2$ together with the classes 
$K_i$ and $d_{i,j}$ for $1 \leq i < j \leq n$.

(2) The tautological ring of $\CM_{g,n}^{rt}$ is 
the $\QQ$-subalgebra of $A^*(\CM_{g,n}^{rt})$ 
generated by $\ka_i$ for $1 \leq i \leq g-2$, the divisor classes $K_i, d_{i,j}$ and $D_I$  
for $1 \leq i < j \leq n$ and subsets $I$ of the set $\{1, \dots, n\}$ having at least 3 elements. 
\end{lem}

The Hodge bundle $\E$ on the moduli space also defines tautological classes.
Recall that the bundle $\E$ on $\CM_g$ is the locally free sheaf of rank $g$ whose fiber over a moduli point 
$[C] \in \CM_g$ is the space of holomorphic differentials on $C$.
Chern classes of the Hodge bundle are also natural cycles.
The Grothendieck-Riemann-Roch computation by Mumford \cite{M} shows that 
lambda classes are expressed in terms of kappa classes and belong to the tautological ring.
For a fixed curve $C$ the connection between tautological classes on the Fulton-MacPherson compactification of $C^n$ 
and the space of curves with rational tails are explained in our previous works \cite{T1,T2,T3}. 
Based on the same picture we can see the connection between tautological classes on $\CM_{g,n}^{rt}$ and $\CC_g^n$.
The idea is to view $\CM_{g,n}^{rt}$ as the relative Fulton-MacPherson 
compactification of the space $\CC_g^n$ over $\CM_g$ for $g \geq 2$. 
In genus one we need to consider the base $\CM_{1,1}$ instead.
To get a uniform description we assume that $g>1$ below.
From this picture one can think of the divisor classes 
$D_I$ for $|I| \geq 3$ as exceptional divisors which appear in the process of blow-ups.
Notice that the pull-back homomorphism identifies the Chow ring of $\CC_g^n$ with a subalgebra of
$A^*(\CM_{g,n}^{rt})$.
We use this identification and view $R^*(\CC_g^n)$ as a subalgebra of $R^*(\CM_{g,n}^{rt})$. 
The following statement follows from the discussion above:

\begin{prop}\label{FM}
The tautological ring $R^*(\CM_{g,n}^{rt})$ of $\CM_{g,n}^{rt}$ is an algebra over $R^*(\CC_g^n)$ generated by the 
divisor classes $D_I$ for subsets $I$ of $\{1, \dots , n\}$ with at least 3 elements. 
\end{prop}

Notice that we do not need to involve the divisor classes $D_I$ for subsets $I$ with 2 elements.
This is because the pull-back of the divisor $d_{i,j}$ via the natural map $\CM_{g,n}^{rt} \too \CC_g^n$ is the sum $\sum_{i,j \in I} D_I$.
We will call a divisor $D_I$ for a subset $I$ with $|I| \geq 3$ an \emph{exceptional divisor}. 

\subsection{Tautological groups in top degrees and the vanishings}\label{socle}

According to the result of Looijenga \cite{L} the tautological group of $\CC_g^n$ is at most one dimensional in degree $g-2+n$ and it vanishes in higher degrees.
The non-vanishing of $\ka_{g-2}$ on $\CM_g$ by Faber \cite{F2} established the one dimensionality.
Similar statements are true for the moduli space $\CM_{g,n}^{rt}$.
They are proved by Faber and Pandharipande \cite{FP2} and Graber and Vakil \cite{GV2}.
In both cases the localization formula for the virtual class is used to reduce the question to the result of Looijenga and the one dimensionality of $R^{g-2}(\CM_g)$.

\subsection{Intersection pairings}

The one dimensionality of the tautological groups of $\CC_g^n$ and $\CM_{g,n}^{rt}$ in top degrees enables us to define 
a pairing on tautological classes.
It is possible to relate intersection matrices of the pairings for $\CC_g^n$ and $\CM_{g,n}^{rt}$.
This connects the Gorenstein quotients of tautological algebras and leads to a conjectural dictionary between tautological relations on these spaces.
Here we give a brief description of the method and refer to \cite{T3} for the details. 
However we warn the reader that some of definitions are slightly different.
We introduce a collection of monomials which additively generate tautological groups of $\CM_{g,n}^{rt}$.
These elements are called \emph{standard monomials}.
To define them we need to associate a directed graph $\G$ to any non-zero monomial 
\begin{equation}\tag{1}\label{vec}
v:=a(v) \cdot D(v)
\end{equation}
which is a product of a class $a(v)$ in $R^*(\CC_g^n)$ 
with a product $$D(v):=\prod_{r=1}^m D_{I_r}^{i_r}$$ 
such that $i_r>0$ for $1 \leq r \leq m$.

\begin{dfn}
The directed graph $\G$ associated to the monomial $v$ in \eqref{vec} consists of a collection of vertices and edges.
Vertices of the graph $\G$ correspond to the subsets $I$ when $D_I$ is a factor of $D$. 
There is an edge from a vertex $I$ to a vertex $J$ when $J$ is a proper subset of $I$ and it is maximal with this property. 
Every minimal vertex is called a \emph{root} of $\G$.
\end{dfn}

Notice that the graph $\G$ only depends on the monomial $D$.
It is straightforward to see that there is no loop in the resulting graph.
The graph $\G$ can be disconnected in general.
The degree of a vertex $I$ is the number of outgoing edges.
To every non-zero monomial 
$v \in R^*(\CM_{g,n}^{rt})$ as in \eqref{vec} 
we associate a subset $S$ of the set 
$\{1, \dots , n\}$ as follows:
Consider the associated graph $\G$.
Denote by $J_1, \dots, J_s$ the set of roots of $\G$.
Let $\al_r \in J_r$ be the smallest element.
We define the set $S$ as follows:
\begin{equation}\tag{2}\label{set}
S:=\{\al_1, \dots , \al_s\} \cup (\cap_{r=1}^m I_r^c).
\end{equation}

\begin{dfn}\label{st}
Let $v$ be a non-zero monomial as in \eqref{vec} and $\G$ be its associated graph.
We say that $v$ is standard if 
$a(v) \in R^*(\CC_g^{S})$ and 
for every vertex $I_r$ of $\G$ the following inequality holds:
$$i_r \leq |I_r|-|\cup_{I_s \subset I_r} I_s|+\deg(I_r)-2.$$
\end{dfn}

We define a total preorder on the collection standard monomials as follows:

\begin{dfn}\label{order}
Let $I,J$ be two subsets of the index set $\{1, \dots , n\}$. 
We say that $I<J$ when $|J|<|I|$ or when $|I|=|J|$ and $I \neq J$.
This induces a preorder on the set of exceptional divisors.
We say that $D_I<D_J$ when $I<J$.
This induces a preorder on the set of all monomials using 
Lexicographic order.
\end{dfn}

\begin{prop}\label{three}
Tautological groups of $\CM_{g,n}^{rt}$ are additively generated by standard monomials. 
\end{prop}

\begin{proof}
The statement follows from a collection of relations.
These relations can be used to write elements in terms of standard monomials. 
They can be divided into two classes:

\begin{itemize}
\item
Let $I$ be a subset of the set $\{1, \dots , n\}$ with at least 3 elements.
For any $i,j \in I$ and $k \in \{1, \dots, n\} \setminus I$ we have the following relations:
\begin{equation}\tag{3}\label{exdiv}
(d_{i,j}+K_j) \cdot D_I=0, \D (d_{i,k}-d_{j,k}) \cdot D_I=0.
\end{equation}
\item
Let $1 \leq i_1 < \dots < i_{k+1} \leq n$ be a sequence of numbers.
To every such sequence associate a polynomial as follows:
$$P(t):= \prod_{i=2}^{i_1} \left(t+d_{1,i} \right) \cdot \prod_{s=1}^k \left(t+d_{1,i_s+1} \right).$$
Consider the subsets $I_r=\{1, \dots, i_{k+1}\}$ and $I_s=\{i_s+1, \dots, i_{s+1}\}$ of the set $\{1, \dots, n\}$ for $1 \leq s \leq k$.
We have the following relation:
\begin{equation}\tag{4}\label{vertex}
P \left(-\sum_{I_r \subseteq I} D_I \right) \cdot \prod_{s=1}^k D_{I_s}=0.
\end{equation}
\end{itemize}

Similar relations were found in previous works by studying intersection rings of blow-ups.
In \cite{T3} we have shown that all these relations can be directly proven on the moduli space as well. 
The first class of relations follow from the well-known formula for $\psi$-classes in genus zero.
The second class of relations are proven by showing that any monomial appearing in the expansion vanishes.
This follows from the vanishing of products of non-intersecting boundary divisors.
We now briefly explain how the statement follows from these relations.
In Definition \ref{order} we defined a preorder on the collection of all monomials. 
Let $v \neq 0$ be the smallest monomial which can not be written as a linear combination of standard monomials.
Assume that $v$ has the product form \eqref{vec}.
Denote by $\CG$ its associated graph.
From relations of the form \eqref{exdiv} we may assume that the index set $S$ satisfies the condition given in \eqref{set}.
By our assumption there is a vertex $I_r$ of the graph $\CG$ for which 
\begin{equation}\tag{5}\label{monomial}
i_r \geq j_r:=|I_r|-|\cup_{I_s \subset I_r} I_s|+\deg(I_r)-1.
\end{equation}
Let $k$ be the degree of the vertex $I_r$ in the graph $\CG$.
This corresponds to $k$ disjoint subsets $I_s$ of $I_r$ for $1 \leq s \leq k$.
From this data we get a relation of the form \eqref{vertex}.
In the expansion of this relation there is a unique term of the form $(-1)^{j_r}D_{I_r}^{j_r} \cdot  \prod_{s=1}^k D_{I_s}$.
By our assumption $i_r \geq j_r$.
After possibly multiplying this relation with an appropriate monomial we can write $v$ as a sum of monomials which are strictly less than $v$.
This contradicts our assumption about $v$ and proves the statement. 
\end{proof}

\subsection{The filtration of the tautological ring}

There is a natural way to define a decreasing filtration on the tautological ring of $\CM_{g,n}^{rt}$.
This filtration consists of $g-2+n$ steps. 
We prove a vanishing statement in terms of this filtration.
This vanishing is our key tool in relating tautological relations on $\CC_g^n$ and $\CM_{g,n}^{rt}$.
The definition of a similar filtration was formulated after a question by Looijenga in connection with our work \cite{T1} in genus one.

\begin{dfn}\label{fil}
Let $v$ be a standard monomial as given in \eqref{vec} and 
$J_1, \dots, J_s$ be roots of the associated graph.
The integer $p(v)$ is defined as follows:
$$p(v):=\deg a(v)+\sum_{r=1}^s |J_r|-s.$$
The subspace 
$\CF^p R^k(\CM_{g,n}^{rt})\subset R^k(\CM_{g,n}^{rt})$ is defined to be the 
$\QQ$-vector space generated by standard monomials $v$
of degree $k$ satisfying $p(v) \geq p$.
\end{dfn}


\begin{dfn}
Let $v,w$ be standard monomials.
We say that $w \ll v$ if every factor of $D(w)$ is less than every factor of $D(v)$.
We also make the convention that $w \ll v$ holds when $D(w)=1$.
\end{dfn}

\begin{prop}\label{triangle}
Let  $v \in R^d (\CM_{g,n}^{rt})$ and $w \in \CF^p R^*(\CM_{g,n}^{rt})$ be such that $v \ll w$.
If $p+d > g-2+n$ then the intersection product $v \cdot w$ is zero.
\end{prop}

\begin{proof}
We may assume that $v \cdot w:=a \cdot D$ is a standard monomial.
Denote by $J_1, \dots, J_s$ the collection of roots of its associated graph.
Consider the factorization of the $D$ part of $w$:
$$D(w):=\prod_{r=1}^m D_{I_r}^{i_r}.$$
The condition in \ref{st} gives a bound for the power $i_r$ of $D_{I_r}$
for each $1 \leq r \leq m$.
Combining these conditions together with the assumption $p+d>g-2+n$
gives that 
$$\deg(a)>g-2+n-\sum_{i=1}^s |J_i|+s+m.$$
We can write the product $a \cdot D$ as $b \cdot D$ 
for some $b$ which is the pull-back of a class from $R^*(\CC_g^k)$, where 
$$k=n-\sum_{i=1}^s |J_i|+s.$$
This means that the degree of $b$ is bigger than
$g-2+k+m$, which is at least $g-2+k$.
The element $b$ vanishes according to the vanishing \cite{L} proven by Looijenga.
\end{proof}

\begin{cor}\label{Fil}
The group $\CF^p R^*(\CM_{g,n}^{rt})$ vanishes when $p > g-2+n$.
\end{cor}

\begin{proof}
The statement follows from Proposition \ref{triangle} when $v=1$. 
\end{proof}

\begin{conj}\label{rt}
The space of relations in $R^*(\CM_{g,n}^{rt})$ is generated by relations in $R^*(\CC_g^n)$ together with the vanishing of:

\begin{itemize}
\item
All relations of the form $x \cdot D_I$, where $I \subseteq \{1, \dots, n\}$
is a subset with at least 3 elements and $x$ is of the form 
$K_i+d_{i,j}$ or $d_{i,k}-d_{j,k}$ for some $i,j \in I$ and $k \in \{1, \dots, n\} \setminus I$. 

\item
All products $D_I \cdot D_J$ for non-intersecting divisors $D_I$ and $D_J$.
\end{itemize}

\end{conj}

\subsection{Gorenstein quotients}

\begin{dfn}
The Gorenstein quotient $G^*(\CM_{g,n}^{rt})$, (resp. $G^*(\CC_g^n)$) is defined as the quotient of
the tautological ring $R^*(\CM_{g,n}^{rt})$, (resp. $R^*(\CC_g^n)$) 
modulo elements which define the zero map via the intersection pairings.
The filtration defined in \ref{fil} induces a filtration on $G^*(\CM_{g,n}^{rt})$ in a natural way.
\end{dfn}

We have identified $R^*(\CC_g^n)$ with a subalgebra of 
$R^*(\CM_{g,n}^{rt})$.
Under this identification we can view $G^*(\CC_g^n)$ as a subalgebra of 
$G^*(\CM_{g,n}^{rt})$.
This follows from the following lemma:

\begin{lem}
The natural inclusion of $R^*(\CC_g^n)$ into $R^*(\CM_{g,n}^{rt})$ induces an 
injection from $G^*(\CC_g^n)$ to $G^*(\CM_{g,n}^{rt})$.
\end{lem}

\begin{proof}
Let $x \in R^*(\CC_g^n)$ be an element of degree $k$ which pairs to zero with all elements of
degrees $g-2+n-k$ in $R^*(\CC_g^n)$.
We need to show that $x$ pairs to zero with all elements $y$ of degrees $g-2+n-k$ in $R^*(\CM_{g,n}^{rt})$.
We may assume that $y:=a(y) \cdot D(y)$ is standard. 
The product $x \cdot y$ vanishes if $D(y)=1$ by our assumption.
We prove the same statement when $D(y) \neq 1$.
If we define $v:=x \cdot a(y)$ and $w:=D(y)$ then it follows immediately that
$v$ and $w$ satisfy the conditions given in Proposition \ref{triangle}.
Therefore the product $x \cdot y$ vanishes.
\end{proof}

Recall that the tautological group of $\CC_g^n$ is one dimensional in degree $g-2+n$ and it vanishes in higher degrees.
Intersection product induces the following pairings among tautological groups:
$$R^k(\CC_g^n) \times R^{g-2+n-k}(\CC_g^n) \too \QQ, \D 0 \leq k \leq g-2+n.$$
By restricting to the Gorenstein quotient all these pairings become perfect.
We fix an involution $*$ on $G^*(\CC_g^n)$ which switches elements in degrees $k$ and $g-2+n-k$ for every $0 \leq k \leq g-2+n$.
This involution induces an involution on $G^*(\CM_{g,n}^{rt})$ as well.
It is defined as follows:

\begin{dfn}
Let $v \in G^k(\CM_{g,n}^{rt})$ be a standard monomial as in Definition \ref{st}.
The dual element $v^*$ is an element in $G^{g-2+n-k}(\CM_{g,n}^{rt})$ defined as:
$$v^*:=a(v)^* \cdot \prod_{r=1}^m D_{I_r}^{j_r},$$
where 
$$j_r:=|I_r|-|\cup_{I_s \subset I_r} I_s|+\deg(I_r)-1-i_r.$$

\end{dfn}

Notice that the defined involution gives a one to one correspondence between standard monomials in degrees $k$ and $g-2+n-k$.
There is a natural way to formulate a connection between $G^*(\CC_g^n)$ and $G^*(\CM_{g,n}^{rt})$ as in Conjecture \ref{rt}.

\begin{thm}\label{RT}
The analogue of Conjecture \ref{rt} holds for the Gorenstein quotients of tautological algebras.  
\end{thm}

\begin{proof}
Let $0 \leq k \leq g-2+n$ and consider the intersection pairing 
among tautological classes in degrees $k$ and $g-2+n-k$.
We consider the set of standard monomials of degree $k$ and their duals.
According to Proposition \ref{triangle} there is a block structure on the intersection matrix.
All blocks below the diagonal vanish.
More precisely, 
let $v_1, v_2$ be standard monomials in $R^k(\CM_{g,n}^{rt})$ satisfying the condition $D(v_1)<D(v_2)$.
Then the intersection product $v_1 \cdot v_2^*$ can be written as $v \cdot w$ for some $v$ and $w$
satisfying the condition given in Proposition \ref{triangle}.
It therefore vanishes.
From Lemma \ref{int_num} it follows that square blocks along the main diagonal are intersection matrices of the pairings of 
$\CC_g^{S}$ for various subsets $S$ of the set $\{1, \dots, n\}$ 
up to a constant depending only on the graphs associated to standard monomials:
Let $v_1$ and $v_2$ be two elements with the property $D(v_1)=D(v_2)$.
Denote by $\G$ their associated graph. 
The intersection number
$$v_1 \cdot v_2^* \in R^{g-2+n}(\CM_{g,n}^{rt}) \cong \QQ$$
and the number
$$a(v_1) \cdot a(v_2)^* \in R^{g-2+|S|}(\CC_g^{S}) \cong \QQ$$
differ by $(-1)^{\epsilon}(2g-2)^{n-|S|}$, where 
$$\epsilon=|\cup_{r=1}^m I_r|+\sum_{i \in V(\G)}\deg(i).$$
In the identification of the tautological group of $\CC_g^n$ in top degree with $\QQ$ we take the generator 
$\ka_{g-2} \cdot \prod_{i=1}^n K_i$.
We also identify $R^*(\CC_g^n)$ with a subring of $R^*(\CM_{g,n}^{rt})$ as usual.
These matrices are all invertible since we are working with Gorenstein quotients.
This shows that there is no more relation among standard monomials.
\end{proof}

\begin{cor}
The tautological ring of $\CM_{g,n}^{rt}$ is Gorenstein if and only if $R^*(\CC_g^m)$ is Gorenstein for all $m \leq n$.
\end{cor}

\begin{proof}
Let $0 \leq k \leq g-2+n$ be an integer. 
The proof of Theorem \ref{RT} shows that the intersection matrix $M$ of the pairing
$$R^k(\CM_{g,n}^{rt}) \times R^{g-2+n-k}(\CM_{g,n}^{rt}) \too \QQ$$ 
has a triangular structure. We know that square blocks along the main diagonal are intersection matrices of the pairings
$$R^l(\CC_g^m) \times R^{g-2+m-l}(\CC_g^m) \too \QQ,$$ 
where $m \leq n$ and $0 \leq l \leq g-2+m$ are integers.
The matrix $M$ is invertible if and only if all of its square blocks are invertible.
The statement follows since there is exactly one square block associated to the pairing
$$R^k(\CC_g^n) \times R^{g-2+n-k}(\CC_g^n) \too \QQ,$$ 
corresponding to all standard monomials $v$ of degree $k$ with $D(v)=1$.
\end{proof}

\begin{lem}\label{int_num}
Let $v$ be a monomial of degree $g-2+n$ as in \eqref{vec}. 
Consider its associated graph $\CG$ and the index set $S$ as in \eqref{set}. 
Assume that $a(v)$ is a tautological class of degree $g-2+|S|$ on $\CC_g^S$ and for every vertex $I_r$ of $\CG$ the following equality holds:
$$i_r=|I_r|-| \cup_{ I_s \subset I_r} I_s|+\deg (I_r)-1.$$
Then we have the following identity:
$$v=(-1)^{\epsilon} a(v) \cdot \prod_{r=1}^s \Delta_{J_r}, \D \epsilon=|\cup_{r=1}^m I_r|+\sum_{i \in V(\G)}\deg(i).$$
Here $\Delta_J$ is the diagonal class associated to the subset $J$ on $\CC_g^n$.
\end{lem}

\begin{proof}
The statement is proven by induction on the number of vertices of the graph $\CG$.
It is obviously true for the empty graph.
Let $J_1, \dots, J_s$ be roots of the graph $\CG$ and let $\al_r \in J_r$ be the smallest element. 
For simplicity we assume that $J_r$ is a maximal vertex of the graph $\CG$.
From \eqref{vertex} the following relation holds:
$$R_r:=\prod_{\al_r \neq j \in J_r} (d_{\al_r,j}-\sum_{J_r \subseteq J} D_J)=0.$$
Here we assume that $k=0$ and the index set $I_0$ is equal to $J_r$.
Consider the following vanishing:
$$a(v) \cdot  R_r \cdot \prod_{1 \leq t \neq r  \leq s} D_{J_t}=0.$$
We show that in the resulting expansion of this relation only two terms survive.
This follows from the following observation: 
Let $J_r \subseteq J$ and assume that the following product
\begin{equation}\tag{6}\label{roots}
a(v) \cdot  D_J \cdot \prod_{1 \leq t \neq r  \leq s} D_{J_t}
\end{equation}
is not zero. Notice that for each $t$ different from $r$ two subsets $J$ and $J_t$ have empty intersection.
We may assume that $a(v) \in R^*(\CC_g^T)$, where $T$ is the complement of the set 
$$T:=J \setminus  \{ \al_r \} \bigcup_{1 \leq r  \neq t \leq s}\left( J_t \setminus \{ \al_t \} \right).$$
From the definition of the powers $i_r$ we conclude that the following identity holds for every vertex $I_r$ of the graph $\CG$:
$$\sum_{I_k} i_k=|I_r|-1,$$
where the sum is taken over all vertices $I_k$ to which there is a directed path from $I_r$ in the graph.
This implies that the degree of $a(v)$ equals to $g-2+n-\sum_{r=1}^s |J_r|+s$, 
which is strictly larger than $g-2+|T|$ if $J_r \subset J$ is a proper subset.
Therefore the non-vanishing of \eqref{roots} forces the equality $J_r=J$.  
As a result we get the identity:
\begin{equation}\tag{7}\label{final}
a(v) \cdot D_{J_r}^{i_r} \cdot \prod_{1 \leq t \neq r  \leq s} D_{J_t}=(-1)^{|J_r|} a(v) \cdot \prod_{\al_r \neq j \in J_r} d_{\al_r,j} \cdot \prod_{1 \leq t \neq r  \leq s} D_{J_t}.
\end{equation}
Using this relation we are able to reduce the question to a graph with smaller number of vertices.
When the vertex $J_r$ is not a maximal vertex we proceed along the same lines.
The difference is that for the general case the relation \eqref{vertex} is slightly more complicated.
\end{proof}

Our results give a different proof of the statements in Section \ref{socle} using the results of Faber and Looijenga:

\begin{prop}
The tautological group of $\CM_{g,n}^{rt}$ is one dimensional in degree $g-2+n$.
It vanishes in higher degrees.
\end{prop}

\begin{proof}
Let $v:=a(v) \cdot D(v)$ be a standard monomial of degree at least $g-2+n$ as in \eqref{vec}.
Let $\CG$ be the graph associated to the monomial $D(v)$ and denote its roots by $J_1, \dots, J_s$.
The index set $S$ is defined as in \eqref{set}.
From the inequalities:
$$g-2+n \leq \deg a(v)+ \deg D(v) \leq g-2+n-\sum_{r=1}^s |J_r|+s + \deg D(v),$$
it follows that $$\deg D(v) \geq \sum_{r=1}^s |J_r|-s.$$
On the other hand the conditions given in Definition \ref{st} imply the strict inequality
$$\deg D(v) < \sum_{r=1}^s |J_r|-s,$$
and therefore the graph $\CG$ should be empty.
This means that $v$ is simply the pull-back of a tautological class from the moduli space $\CC_g^n$.
The statements follow immediately from the results in \cite{F2,L} on the tautological groups of $\CC_g^n$. 
\end{proof}

\section{Final remarks}

The result of Petersen \cite{P} confirms Conjecture \ref{rt}.
His approach is based on a study of Fulton-MacPherson spaces for families and proves analogue results in more general settings.
In \cite{Pi} Pixton introduces a large collection of tautological relations on $\MM_{g,n}$.
He conjectures that these relations give a complete set of generators among tautological classes.
We can restrict Pixton's relation on $\CM_{g,n}^{rt}$ and consider their push-forwards to $\CC_g^n$. 

\begin{Q}
Can one relate Pixton's relations on $\CM_{g,n}^{rt}$ in terms of his relations on 
$\CC_g^n$ as described in Conjecture \ref{rt}?
\end{Q}

\bibliographystyle{amsplain}
\bibliography{mybibliography}
\end{document}